\documentclass[10pt]{amsart}

\usepackage{mathpazo}
\usepackage[mathpazo]{flexisym}
\usepackage{breqn}

\usepackage{verbatim}
\usepackage{graphicx}
\usepackage{xcolor}
\usepackage{amsfonts,amsmath,latexsym,amssymb,amsthm}
\usepackage{hyperref}

\newcounter{theoremcounter}
\newcounter{lemmacounter}

\newcounter{dummycounter}
\newcounter{propcounter}

\newcounter{name}
\newcounter{quescounter}

\newcounter{emptycounter}

\newtheorem{theorem}[theoremcounter]{Theorem}
\newtheorem{principle}[name]{Principle}
\newtheorem{question}[quescounter]{Question}

\newtheorem{lemma}[lemmacounter]{Lemma}

\newtheorem{proposition}[propcounter]{Proposition}

\newcounter{eqncounter}

\numberwithin{equation}{eqncounter}

\def\IQ{\mathbb Q}

\newcommand{\Mod}[1]{\ (\mathrm{mod}\ #1)}

\def\Qbar{\overline{\IQ}}

\def\c{c}
\def\cone{c_1}
\def\ctwo{c_2}
\def\cthree{c_3}

\def\Lin{L}

\def\Gal{\mathop{{\rm Gal}}\nolimits}

\title{Small generators of abelian number fields}

\begin{document}

\author{Martin Widmer}

\address{Department of Mathematics\\ 
Royal Holloway, University of London\\ 
TW20 0EX Egham\\ 
UK}

\email{martin.widmer@rhul.ac.uk}

\date{\today}

\subjclass[2010]{Primary 11R04; 11G50   Secondary 11R29; 11R06}

\keywords{Weil height, small generators, Ruppert's conjecture, abelian number fields, Linnik's Theorem}

\maketitle

\begin{abstract}
We show that for each abelian number field $K$ of sufficiently large degree $d$ there exists an element $\alpha\in K$ with 
$K=\IQ(\alpha)$ and absolute Weil height $H(\alpha)\ll_d |\Delta_K|^{1/2d}$ , where $\Delta_K$ denotes the discriminant of $K$.
This answers a question of Ruppert from 1998 in the case of abelian extensions of sufficiently large degree.
We also show that the exponent $1/2d$ is best-possible when $d$ is even.    
\end{abstract}

\section{Introduction}
In this article we answer Ruppert's question on small generators of number fields for abelian fields of sufficiently large degree.
We also show that Ruppert's proposed bound is sharp when the degree is even.

Let $K$ be a number field, and for $\alpha\in K$ let 
\begin{alignat*}1
H(\alpha)=\prod_{v\in M_K}\max\{1,|\alpha|_v\}^{\frac{d_v}{[K:\IQ]}}
\end{alignat*}
be the absolute multiplicative Weil height of $\alpha$. Here $M_K$ denotes the set of places of $K$ and 
for each place $v$ we choose the unique representative $|\cdot|_v$ that either extends the usual Archimedean absolute value
on $\IQ$ or a usual $p$-adic absolute value on $\IQ$, and $d_v=[K_v:\IQ_v]$ denotes the local degree at $v$. By the product formula 
$H(\alpha)$ is independent of the number field $K$ containing $\alpha$, and thus
$H(\cdot)$ is a well-defined function on the algebraic numbers $\Qbar$.
For an extension $K/F$ of number fields we write 
\begin{alignat*}1
\delta(K/F)=\min \{H(\alpha); K=F(\alpha)\}.
\end{alignat*}
Northcott's Theorem implies that the above minimum exists, and thus  $\delta(K/F)$ is well-defined.
We write $\Delta_K$ for the discriminant of $K$. 
In 1998 Ruppert \cite[Question 2]{Rupp} proposed the following question.
\begin{question}\label{quest: Ruppert1}[Ruppert, 1998]
Let $d>1$ be an integer. Is it true that
\begin{alignat}1\label{ineq: Ruppert1} 
\delta(K/\IQ)\ll_d|\Delta_K|^{\frac{1}{2d}}
\end{alignat}
for every number field $K$ of degree $d$?
\end{question} 

Ruppert showed that for $d=2$ the answer is yes. 
In \cite{VaalerWidmer} an affirmative answer to Ruppert's question  was  given for all
number fields $K$ that have a real embedding, confirming (\ref{ineq: Ruppert1}) for all odd $d$.
In the same article it was also shown that (\ref{ineq: Ruppert1}) holds, provided the Riemann-hypothesis
for the Dedekind-zeta function of the Galois closure $K^{(G)}$ of $K/\IQ$ holds true. 


Pierce, Turnage-Butterbaugh and Wood \cite[Theorem 8.2]{PTW} have  shown that   (\ref{ineq: Ruppert1}) 
holds for $100\%$ (when enumerated by modulus of their discriminant) of all number fields in various families of degree $d$-extensions with prescribed Galois group and certain conditions on the  ramification types\footnote{Some of their results are conditional on the strong Artin conjecture  (but not  the Riemann-hypothesis for the corresponding Dedekind-zeta function).}.

Very recently  Akhtari, Vaaler,
and the author 
\cite[Theorem 1.3]{AkhtariVaalerWidmer2023}  showed that   (\ref{ineq: Ruppert1}) 
can only fail for  totally complex fields $K$ that are Galois over their maximal totally real subfield $F$.
Hence, all these exceptional fields do have a subfield of index $2$, namely the fixed field of the group generated by complex conjugation. 
This implies that, apart from the  families of cyclic number fields \cite[Theorem 3.3, (1)]{PTW},  the  unconditional parts of the ``almost all'' result \cite[Theorem 8.2]{PTW}  
are superseded\footnote{The families of even degree considered in \cite[Theorem 8.2]{PTW}  are (certain) $d$-extensions $K$ with  $\Gal(K^{(G)}/\IQ)$ isomorphic to either $A_d$, $S_d$,  a cyclic group,
or a simple group (the latter being conditional on the strong Artin conjecture). 
Each subgroup of $A_d$ or $S_d$ of index $d$ is some $A_{d-1}$ or $S_{d-1}$ respectively, and 
there is neither  a proper intermediate group between $A_{d-1}$ and $A_{d}$ nor between $S_{d-1}$ and $S_{d}$. Hence, these fields $K$ cannot have a subfield of index $2$
if $d>2$.}
by the recent  pointwise result  \cite[Theorem 1.3]{AkhtariVaalerWidmer2023}.
 
Regarding cyclic, and more generally abelian, number fields, we have the following pointwise result.
\begin{theorem}\label{Thm:2}
Let $\Lin$ be the effectively computable constant from (\ref{def:nuexplicit}), and suppose $d\geq 4\Lin$.
Then, apart from finitely many exceptional fields, we have  
$$\delta(K/\IQ)\leq (25|\Delta_K|)^{\frac{1}{2d}}$$
for each  abelian number field $K$ of degree $d$. 
\end{theorem}
The proof is a combination of well-known facts, classical tools, and  \cite[Theorem 1.4]{VaalerWidmer}, 
 which yields (\ref{ineq: Ruppert1}) provided there exists  a product of distinct degree one primes of size between  $|\Delta_K|^{1/2}$ and $\gamma |\Delta_K|^{1/2}$
 (for some $\gamma=\gamma_d>1$).

For abelian number fields a sufficient condition for the splitting of a prime $p$ can be given in terms of a  congruence condition 
modulo the conductor $\frak{f}$ of $K$. The latter is related to the discriminant by Hasse's conductor-discriminant formula, and one can show that $\mathfrak{f}\leq |\Delta_K|^{2/d}$. 
Linnik's Theorem provides a prime $p$ 
in the required residue class  modulo $\mathfrak{f}$ and satisfying $p\leq \mathfrak{f}^L$, where $L$ is an absolute constant (usually called Linnik's constant).
This upper bound is sufficient as long as $d\geq 4L$.
We also need $p>|\Delta_K|^{1/2}$ but a  minor adjustment  to the proof of Linnik's Theorem  allows one to take care of this
additional constraint.

For abelian fields of small degree better bounds for the least splitting prime $p$ were obtained by
Pollack \cite{Pollack2014} (see also  \cite{GeMilinovichPollack}  for an extension to certain non-abelian number fields).
The proof is based on the observation that 
a lower bound on the number of  ideals  of bounded norm  furnishes an upper bound on the smallest splitting prime.
Unfortunately, it is unclear how to use this method to establish a splitting prime $p$ in the required interval
$|\Delta_K|^{1/2}< p\ll_d |\Delta_K|^{1/2}$.

The constant $\Lin$ in Theorem \ref{Thm:2} is expressed in terms of the constants $\c,\cone,\ctwo,\cthree$ that 
arise from the three fundamental principles Linnik's Theorem is based upon (cf. Section \ref{sec: primes}).
These four constants are all explicitly computable and are subject to further improvement. However, getting good values requires rather tedious computations,
which come at the expense of clarity and simplicity of the proof. But even the strongest methods (as developed in \cite{Heath-Brown1991}
and \cite{Xylouris2011}) would at very best lead to the bound $\Lin=5$, leaving the cases $2<d<20$ still open. Therefore, we have not attempted to present explicit values for 
$\c,\cone,\ctwo$, and $\cthree$.

It is also worthwhile mentioning that the bound (\ref{ineq: Ruppert1})
is stable under taking composite fields with coprime discriminants.
To see this let 
$\alpha_i\in K_i$ with 
$$\delta(K_i/\IQ)=H(\alpha_i)\ll_{d_i}|\Delta_{K_i}|^{\frac{1}{2d_i}} \quad (1\leq i \leq 2).$$ 
There exists a generator $\alpha=m_1\alpha_1+m_2\alpha_2$ of $K=K_1K_2$ with rational integers $0\leq m_i<d$ (see, e.g., \cite[Lemma 3.3]{art1}).  If $(\Delta_{K_1},\Delta_{K_2})=1$ then it follows (cf. \cite[Theorem 4.9]{Narkiewiczbook}) that
$d=[K:\IQ]=d_1d_2$ and\footnote{In fact, $d=d_1d_2$ implies that already  $\alpha=\alpha_1+\alpha_2$ is a generator of $K$, as was shown
by  Isaac (cf. \cite[Theorem 1.3]{Weintraub2011}).} 
$$|\Delta_K|=|\Delta_{K_1}^{d_2}\Delta_{K_2}^{d_1}|,$$ and thus 
$$\delta(K/\IQ)\leq H(\alpha)\ll_ {d}|\Delta_{K_1}|^{\frac{1}{2d_1}}|\Delta_{K_2}|^{\frac{1}{2d_2}}=|\Delta_{K}|^{\frac{1}{2d}}.$$

Can the exponent $1/2d$ in (\ref{ineq: Ruppert1}) be improved?
A  well-known result of  Silverman \cite[Theorem 1.1]{Silverman} implies that 
\begin{alignat}1\label{ineq: Silverman} 
\delta(K/\IQ)\gg_d|\Delta_K|^{\frac{1}{2d(d-1)}}.
\end{alignat}
Ruppert's \cite[Theorem 1.1]{Rupp} shows that for $d=2$ the reversed inequality holds true as well. 
This suggests the following strengthening of Question \ref{ineq: Ruppert1}, also proposed
by Ruppert \cite[Question 1]{Rupp}.
\begin{question}\label{quest: Ruppert2}[Ruppert, 1998]
Let $d>1$ be an integer. Is it true that
\begin{alignat}1\label{ineq: Ruppert2}
\delta(K/\IQ)\ll_d|\Delta_K|^{\frac{1}{2d(d-1)}}
\end{alignat}
for every number field $K$ of degree $d$?
\end{question}

However, Vaaler and the  author  \cite[Theorem 1.2]{VaalerWidmer2} have shown that there exist infinitely many number fields $K$ of degree $d$ with 
\begin{alignat}1\label{ineq: 0}
\delta(K/\IQ)\geq |\Delta_K|^{\gamma}
\end{alignat}
whenever 
\begin{alignat}1\label{ineq: 1}
\gamma<\left\{
    \begin{array}{lll}
      &1/((b+1)d):&\text{ if }b\leq 3,\\
      &{1}/{(2(b+1)d)}+{1}/{(b^2(b+1)d)}:&\text{ otherwise.}
\end{array} \right.
\end{alignat}
Here $b=b(d)>1$ denotes the smallest 
divisor of $d$.  
This provides a negative answer to Question \ref{quest: Ruppert2} when $d$ is composite.

A negative answer was also given \cite[Theorem 1.3]{VaalerWidmer2} for prime $d>3$, conditional on a weak form of a folklore conjecture about the distribution of number fields. 
Very recently Dubickas  \cite[Theorem 1]{Dubickas22} showed unconditionally that for odd $d$ one can take 
\begin{alignat}1\label{ineq: Dubickas}
\gamma=\frac{d+1}{2d^2(d-1)},
\end{alignat}
hence showing that Ruppert's Question \ref{quest: Ruppert2} has a negative answer also for prime $d>2$. While the bound
(\ref{ineq: 1}) is better for composite $d$ than (\ref{ineq: Dubickas}) Dubickas provides an explicit family
of fields whereas there is no such description for the fields in \cite[Theorem 1.2 and Corollary 4.1]{VaalerWidmer2}.

Our next result  shows that for even $d$ the exponent $1/2d$ in  
(\ref{ineq: Ruppert1}) cannot be replaced  by any smaller value. Furthermore, it
improves the exponents $\gamma$ in (\ref{ineq: 1}) and (\ref{ineq: Dubickas}) when $d$ is composite and $b\neq 3$,
and it provides an explicit description of the fields $K$. 
On the other hand the method in \cite{VaalerWidmer2} (cf.  \cite[Corollary 4.2]{VaalerWidmer2}) also  provides density statements which we do not recover here. 
\begin{theorem}\label{Thm:1}
Let $m$ and $n>1$ be positive integers, and set $d=mn$. 
Then there are infinitely many number fields $K$ of degree $d$ such that
\begin{alignat*}1
{\frac{1}{\sqrt{2d}}}|\Delta_K|^{\frac{1}{2d(n-1)}}\leq \delta(K/\IQ)\leq (2d)^2|\Delta_K|^{\frac{1}{2d(n-1)}}.
\end{alignat*}
More precisely, the above bounds hold for each $K=\IQ(2^{1/m}, (p/q)^{1/n})$ where $p$ and $q$ are primes satisfying $m<p<q<2p$.
\end{theorem}
Theorem \ref{Thm:1} is a simple consequence of an observation of Ruppert (and independently of Masser), and a standard height inequality of Silverman.\\


In \cite{VaalerWidmer2} it was proposed to determine the cluster points\footnote{ What the authors in \cite{VaalerWidmer2} meant by ``cluster point of the set $\left\{\frac{\log\delta(K/\IQ)}{\log|\Delta_K|}; [K:\IQ]=d\right \}$'' is a real number $\eta$ such that for each $\varepsilon>0$ there are infinitely 
many number fields $K$ of degree $d$ with $\left|\frac{\log\delta(K/\IQ)}{\log|\Delta_K|}-\eta\right|<\varepsilon$.} of the set
\begin{alignat*}1
\left\{\frac{\log\delta(K/\IQ)}{\log|\Delta_K|}; [K:\IQ]=d\right \}.
\end{alignat*}
Ruppert  \cite{Rupp} has shown that $1/(2d(d-1))$ is the smallest cluster point, and
for $d=3$ Dubickas has recently shown that  $1/9$ is a cluster point. Theorem \ref{Thm:1} shows that  $1/2d(n-1)$ is a cluster point for every divisor $n>1$ of $d$.


\section*{Acknowledgements}
It is my pleasure to thank  Glyn Harman, for  pointing  out that a lower bound as in Lemma \ref{lemm:piposbound}
can be derived by an easy adaptation of the usual proof of Linnik's theorem,  and  to
Marc Technau, for showing that Proposition \ref{Prop:Linnik} follows easily from  \cite[Proposition 18.5]{IW04}. 
I am also grateful to  Shabnam Akhtari and Jeffrey Vaaler for
many fruitful and interesting discussions on Ruppert's questions and beyond.

\section{Primes of  a given residue class in short intervals}\label{sec: primes}
For $x\geq 1$ and coprime positive integers $a$ and $q$ let

\begin{alignat*}1
\pi(x;q,a)&=\sum_{p\leq x\atop p\equiv a \Mod{q}}1,\\
\psi(x;q,a)&=\sum_{n\leq x\atop n\equiv a \Mod{q}}\Lambda(n),
\end{alignat*}
where in the first sum $p$ denotes a prime, and in the second sum $\Lambda(\cdot)$ denotes the von Mangoldt function.
Noting that 
$$\psi(x;q,a)=\sum_{p\leq x\atop p\equiv a \Mod{q}}\left\lfloor\frac{\log x}{\log p}\right\rfloor \log p,$$ 
and splitting the sum in $p\leq \sqrt{x}$ and $p>\sqrt{x}$ gives the inequalities 
\begin{alignat*}1
\frac{1}{2}\log x\left(\pi(x;q,a)-\pi(\sqrt{x};q,a)\right) \leq \psi(x;q,a)\leq \log x\left(\pi(x;q,a) +\pi(\sqrt{x};q,a)\right).
\end{alignat*}
This in turn yields for $y>x\geq 2$
\begin{alignat}1\label{ineq: pivspsi}
\pi(y;q,a)-\pi(x;q,a)\geq \frac{\psi(y;q,a)}{\log y}-\frac{2\psi(x;q,a)}{\log x}+O\left(\frac{\sqrt{y}}{\log y}\right),
\end{alignat}
where the implicit  constant in the $O$-term is absolute.

As explained in Section \ref{sec: proofThm2}, to prove Theorem \ref{Thm:2} it suffices to prove a positive lower bound for 
$$\pi(y;q,a)-\pi(x;q,a),$$
when $y=\gamma x$ for some absolute $\gamma>1$ and $x$ is of  size $q^{d/4}$. 

Such a lower bound can be deduced by an easy adaptation of the usual proof of Linnik's theorem, provided $d$ is sufficiently big.
We follow Iwaniec-Kowalski \cite{IW04} which itself  borrows significantly from 
Graham's work \cite{Graham81, Graham78}. 
However, to keep track of the constants which determine the threshold $4\Lin$ for the degree in Theorem \ref{Thm:2} we 
need to give some background.

Following \cite{IW04} we write $L_q(s)$ for the product of all Dirichlet $L$-functions associated to
Dirichlet characters $\chi$ modulo $q$ 
\begin{alignat*}1
L_q(s)=\prod_{\chi \Mod{q}}L(s,\chi).
\end{alignat*}
For $\frac{1}{2}\leq \alpha\leq 1$, and $T\geq 1$ we write
\begin{alignat*}1
N_q(\alpha,T)
\end{alignat*}
for the number of zeros of $L_q(s)$, counted with multiplicity, with real part 
$\alpha<\sigma\leq 1$, and modulus of the imaginary part $|t|\leq T$.

The proof of Linnik's Theorem rests on the following three principles which we quote from 
 \cite{IW04}\footnote{In \cite{IW04} the authors state these principles for more general ranges $|t|\leq T$ but only apply them with $T=q$.
The authors \cite[Section 18.1, p. 429]{IW04} remark that for $q$ sufficiently large and  $T\leq \log q$ one can choose $\cone=1/10, \ctwo=3,\cthree=1/2$ (they provide no value for $\c$).
However, following \cite{IW04} we choose $T=q$, and thus these values may not be eligible in our setting.}.

\begin{principle}[The zero-free region]
There is an effectively computable positive constant 
$\cone$ such that $L_q(s)$ has at most one zero  in the region 
\begin{alignat*}1
\sigma\geq 1-\cone/2\log q, \quad |t|\leq q.
\end{alignat*}
The exceptional zero, if it exists, is real and simple, and it is for a real,
non-principal character.
\end{principle}

\begin{principle}[The log-free zero-density estimate]
There are effectively computable positive constants
$\c, \ctwo$ such that for any $\frac{1}{2}\leq \alpha\leq 1$ 
\begin{alignat*}1
N_q(\alpha,q)\leq \c q^{2\ctwo(1-\alpha)}.
\end{alignat*}
\end{principle}

\begin{principle}[The exceptional zero repulsion]
There is an effectively computable positive constant 
$\cthree$ such that if the exceptional zero $\beta_1$ exists, say $L(\beta_1,\chi_1)=0$
with 
\begin{alignat*}1
1-\cone/2\log q\leq \beta_1 <1,
\end{alignat*}
then the function $L_q(s)$ has no other zeros in the region
\begin{alignat*}1
\sigma\geq 1-\cthree\frac{|\log(1-\beta_1)|}{2\log q}, \quad |t|\leq q.
\end{alignat*}
\end{principle}
For the effectively computable constants from Principle 1, 2, and 3 we assume (as we can) throughout  this section that
\begin{alignat*}1
\c,\ctwo>1>\cone,\cthree>0.
\end{alignat*}

We say that the real number $\beta_1$ is an exceptional zero of $L_q(s)$ if there exists a real character $\chi_1$ modulo $q$ for which  $L(\beta_1,\chi_1)=0$, and $\beta_1=1-\delta_1$ with
\begin{alignat*}1
\delta_1\leq \frac{\cone}{2\log q}.
\end{alignat*}

Let 
\begin{alignat*}1
\eta=\frac{\cone}{2\log q}
\end{alignat*}
if the exceptional zero $\beta_1$ does not exist, and let 
\begin{alignat*}1
\eta=\cthree\frac{|\log(2\delta_1\log q)|}{2\log q}
\end{alignat*}
if $\beta_1$ does exist.

We will use \cite[Proposition 18.5]{IW04} which we state as a lemma.

\begin{lemma}\label{lemm:Linnik}\cite[Proposition 18.5]{IW04}
Let $\c,\cone,\ctwo,\cthree$ be the absolute constants from Principles 1, 2, 3.  
Then for $x\geq q^{4\ctwo}$ we have
\begin{alignat*}1
\psi(x;q,a)=\frac{x}{\phi(q)}\left\{1-\chi_1(a)\frac{x^{\beta_1-1}}{\beta_1}+\theta cx^{-\eta/2}+O\left(\frac{\log q}{q}\right)\right\}
\end{alignat*}
where the term involving $\beta_1$ does not exist if $\beta_1$ does not exist, $\eta$ is as above and $|\theta|\leq 4$, and the implied constant is absolute.
\end{lemma}
Next we consider the quantity\footnote{We have additionally included the term $4/\cthree$ as $\nu\geq 4/\cthree$ is used in \cite{IW04} without mentioning.} $\nu$ from \cite[(18.90)]{IW04}, 
but we allow an additional  real parameter  $U$, and we set
\begin{alignat}1\label{def:nu}
\nu=\max\left\{4\ctwo,\frac{4}{\cone},\frac{4}{\cthree},\frac{4\log(2U\c)}{\cthree|\log \cone|}\right\}.
\end{alignat}

\begin{proposition}\label{Prop:Linnik}
Suppose the exceptional zero  $\beta_1$ of $L_q(s)$ exists,
let $U\geq 1$, and let $\gamma\geq 5$. Further,
suppose that $x\geq q^{\nu}$. Then
\begin{alignat*}1
\psi(\gamma x;q,a)-3\psi(x;q,a)\geq \frac{x}{\phi(q)}\left\{\left[\frac{4(\gamma-3)}{3\cone}-\frac{4(\gamma+3)}{U\cone}-\frac{(\gamma-3)}{\log q}\right]\delta_1\log q
+O\left(\frac{\gamma \log q}{q}\right)\right\},
\end{alignat*}
where the implied constant is absolute.
\end{proposition}
\begin{proof}
Let  $\c,\cone,\ctwo,\cthree$ and $\eta$ be as in Lemma \ref{lemm:Linnik}. 
Applying Lemma \ref{lemm:Linnik}
we get
\begin{alignat*}1
\psi(\gamma x;q,a)-3\psi(x;q,a)=&\frac{\gamma x}{\phi(q)}\left\{1-\chi_1(a)\frac{(\gamma x)^{\beta_1-1}}{\beta_1}+\theta_1c(\gamma x)^{-\eta/2}+O\left(\frac{\log q}{q}\right)\right\}\\
&-\frac{3x}{\phi(q)}\left\{1-\chi_1(a)\frac{x^{\beta_1-1}}{\beta_1}+\theta_2 cx^{-\eta/2}+O\left(\frac{\log q}{q}\right)\right\}\\ =&\frac{x}{\phi(q)}\left\{(\gamma-3)-(\gamma^{\beta_1}-3)\chi_1(a)\frac{x^{\beta_1-1}}{\beta_1}+(\theta_1\gamma^{1-\eta/2}-3\theta_2)cx^{-\eta/2}+O\left(\frac{\gamma \log q}{q}\right)\right\}.
\end{alignat*}
Since $\gamma\geq 5$ we have $\left|\gamma^{\beta_1}-3\right |\leq \gamma-3$, and since $|\theta_i|\leq 4$ we get $|\theta_1\gamma^{1-\eta/2}-3\theta_2|\leq 4(\gamma+3)$.
Hence,
\begin{alignat}1\label{est:psi}
\psi(\gamma x;q,a)-3\psi(x;q,a)
\geq \frac{x}{\phi(q)}\left\{(\gamma-3)\left(1-|\chi_1(a)|\frac{x^{\beta_1-1}}{\beta_1}\right)-4(\gamma+3)cx^{-\eta/2}+O\left(\frac{\gamma \log q}{q}\right)\right\}.
\end{alignat}
Now following \cite[p. 441]{IW04} we get 
\begin{alignat}1\label{est:1-chi}
\nonumber 1-|\chi_1(a)|\frac{x^{\beta_1-1}}{\beta_1}= 1-\frac{x^{-\delta_1}}{\beta_1}\geq \beta_1-q^{-\nu\delta_1}
=1-q^{-\nu\delta_1}-\delta_1\geq \frac{\nu \delta_1\log q}{1+\nu \delta_1\log q}-\delta_1\geq 
 \frac{\nu \delta_1\log q}{1+\frac{\nu \cone}{2}}-\delta_1\\
 \geq \frac{4\delta_1}{3\cone}\log q-\delta_1.
 \end{alignat}
Furthermore, following \cite[p. 441]{IW04} and using that $\nu\cthree/4\geq 1$, we get 
\begin{alignat}1\label{est:xeta}
x^{-\eta/2}\leq q^{-\nu\eta/2}=(2\delta_1\log q)^{\nu\cthree/4}\leq (2\delta_1\log q)\cone^{\nu\cthree/4-1}\leq \frac{\delta_1}{U\c\cone}\log q.
 \end{alignat}
Finally, plugging (\ref{est:1-chi}) and (\ref{est:xeta}) into (\ref{est:psi}) yields
\begin{alignat*}1
\psi(\gamma x;q,a)-3\psi(x;q,a) \geq \frac{x}{\phi(q)}\left\{\left[\frac{4(\gamma-3)}{3\cone}-\frac{4(\gamma+3)}{U\cone}-\frac{(\gamma-3)}{\log q}\right]\delta_1\log q
+O\left(\frac{\gamma \log q}{q}\right)\right\}.
\end{alignat*}
\end{proof}

\begin{lemma}\label{lemm:piposbound}
Let  $5\leq \gamma\leq \sqrt{x}$, let $U>3\left(\frac{\gamma+3}{\gamma-3}\right)$, and let 
\begin{alignat}1\label{def:nu}
\Lin_U=\max\left\{4\ctwo,\frac{4}{\cthree},\frac{4\log(2U\c)}{\cone},\frac{4\log(2U\c)}{\cthree|\log \cone|}\right\}.
\end{alignat}

There exists $q_0$ such that for $q\geq q_0$ and $x\geq q^{\Lin_U}$
\begin{alignat*}1
\pi(\gamma x;q,a)-\pi(x;q,a)\gg_{\gamma, U,\cone,q_0} \frac{x}{\phi(q)\sqrt{q}\log \gamma x}.
\end{alignat*}
\end{lemma}
\begin{proof}
From (\ref{ineq: pivspsi}), and using that $q\leq x^{1/\Lin_U}\leq x^{1/4}$, we see that it suffices to show that
\begin{alignat}1\label{ineq: psiposbound}
\frac{\psi(\gamma x;q,a)}{\log \gamma x}-\frac{2\psi(x;q,a)}{\log x}\gg_{\gamma, U,\cone,q_0} \frac{x}{\phi(q)\sqrt{q}\log \gamma x}.
\end{alignat}
Since $\gamma\leq \sqrt{x}$ we have 
\begin{alignat}1\label{eq: psisimp}
\frac{\psi(\gamma x;q,a)}{\log \gamma x}-\frac{2\psi(x;q,a)}{\log x}\geq \frac{1}{\log \gamma x}\left(\psi(\gamma x;q,a)-3\psi(x;q,a)\right).
\end{alignat}
First suppose the exceptional zero $\beta_1$ of $L_q(s)$ exists.
Then (\ref{ineq: psiposbound}) follows immediately from (\ref{eq: psisimp}) and Proposition \ref{Prop:Linnik} upon noticing that $\delta_1\log q\gg q^{-1/2}$
by Siegel's Theorem.

Next suppose the exceptional zero $\beta_1$ does not exists so that $\eta=\cone/(2\log q)$.
Applying Lemma \ref{lemm:Linnik}, and using that $x\geq q^{\frac{4\log(2U\c)}{\cone}}$, yields
\begin{alignat*}1
\psi(\gamma x;q,a)-3\psi(x;q,a)=&\frac{x}{\phi(q)}\left\{(\gamma-3)+(\theta_1\gamma^{1-\eta/2}-3\theta_2)cx^{-\eta/2}+O\left(\frac{\gamma \log q}{q}\right)\right\}\\
\geq & \frac{x}{\phi(q)}\left\{(\gamma-3)-4(\gamma+3)cx^{-\eta/2}+O\left(\frac{\gamma \log q}{q}\right)\right\}\\
\geq & \frac{x}{\phi(q)}\left\{(\gamma-3)-\frac{2(\gamma+3)}{U}+O\left(\frac{\gamma \log q}{q}\right)\right\}\gg_{\gamma,U,q_0} \frac{x}{\phi(q)\sqrt{q}}.
\end{alignat*}
This in conjunction with (\ref{eq: psisimp}) proves (\ref{ineq: psiposbound}).
\end{proof}


\section{Proof of Theorem \ref{Thm:2}}\label{sec: proofThm2}
We will apply Theorem 4.1 from \cite{VaalerWidmer}. Here we state only a special case, sufficient for our purposes.
\begin{theorem}\label{Thm:VaalerWidmer}\cite[Theorem 4.1]{VaalerWidmer}
Let $K$ be a number field of degree $d$, and let $p>|\Delta_K|^{1/2}$ be a prime that splits completely in $K$. 
Then there exists $\alpha\in K$ with $K=\IQ(\alpha)$ and $H(\alpha)\leq p^{1/d}$.
\end{theorem}

If $K$ is an abelian number field of  conductor $\frak{f}$ then any prime $p\equiv 1 \pmod{\frak{f}}$
splits completely in $K$ (cf. \cite[Theorem 8.1]{Narkiewiczbook}).
Thus, to prove Theorem \ref{Thm:2} it suffices to show that, apart from finitely many exceptional fields,
there exists  a prime $p\equiv 1 \pmod{\frak{f}}$ with $ |\Delta_K|^{1/2}<p\leq 5  |\Delta_K|^{1/2}$. 
Let $\c,\ctwo>1>\cone,\cthree>0$ be the absolute constants from Principles 1, 2, 3.
We apply Lemma \ref{lemm:piposbound} with
$\gamma=5$, $U=13$, $a=1$, and $q=\frak{f}$ to conclude that if $\frak{f}$ is sufficiently large
then there exists a prime $p$ that splits completely in $K$ and satisfies $x<p\leq 5x$, provided $x\geq \frak{f}^{\Lin}$ where
\begin{alignat}1\label{def:nuexplicit}
\Lin=\Lin_{13}=\max\left\{4\ctwo,\frac{4}{\cthree},\frac{4\log(26\c)}{\cone},\frac{4\log(26\c)}{\cthree|\log \cone|}\right\}.
\end{alignat}
We note that $\frak{f}\leq |\Delta_K|^{2/d}$ (see \cite[Lemma 9.2.1]{Cohen2000}). Since $d\geq 4\Lin$ 
the hypothesis $x\geq \frak{f}^{\Lin}$ holds for $x=|\Delta_K|^{1/2}$, and thus we conclude that
there exists a splitting prime $p$ with $|\Delta_K|^{1/2}<p\leq 5|\Delta_K|^{1/2}$ whenever $\frak{f}$ is sufficiently 
large. Thus  Theorem \ref{Thm:VaalerWidmer} yields a sufficiently small generator for these fields. 
Finally, as there are only finitely many abelian fields of degree $d$ with conductor below a given bound the proof of Theorem \ref{Thm:2} is complete.

\section{Proof of Theorem \ref{Thm:1}}
Theorem \ref{Thm:1} is an immediate consequence of the following proposition.
\begin{proposition}\label{Prop:1}
Let $m$ and $n>1$ be positive integers, let $F$ be a number field of degree $m=[F:\IQ]$, and set $d=mn$. \\
(1) Let $K$  be a field extension of $F$ of degree $n=[K:F]$. Then we have 
\begin{alignat*}1
\delta(K/\IQ)\geq \left(n^m|\Delta_F|\right)^{-\frac{1}{2m(n-1)}}|\Delta_K|^{\frac{1}{2d(n-1)}}.
\end{alignat*}
(2) Let $K=F((p/q)^{1/n})$ where $p$ and $q$ are primes, $p$ is unramified in $F$ and $p<q<2p$.  Then $[K:F]=n$ and
\begin{alignat*}1
\delta(K/\IQ)\leq 2^{1+\frac{1}{2n}}d^2\delta(F/\IQ)|\Delta_K|^{\frac{1}{2d(n-1)}}.
\end{alignat*}
\end{proposition}
\begin{proof}
By a result of Silverman \cite[Theorem 2]{Silverman} we have 
$$H(\alpha)\geq n^{-\frac{1}{2(n-1)}}N_{F/\IQ}(D(K/F))^{\frac{1}{2d(n-1)}}$$ 
whenever $K=F(\alpha)$. Here $D_{K/F}$ denotes the relative discriminant of $K/F$ and $N_{F/\IQ}(\cdot)$ is the norm
from $F$ to $\IQ$. Since $N_{F/\IQ}(D(K/F))=|\Delta_F|^{-n}|\Delta_K|$ the first part follows immediately.

For the second part we start with an observation of Ruppert \cite{Rupp}: let $\alpha=(p/q)^{1/n}$ and set $M=\IQ(\alpha)$. 
Then  $H(\alpha)=q^{1/n}$ and $(pq)^{n-1}|\Delta_M$.
Hence $H(\alpha)\leq 2^{1/2n}|\Delta_M|^{1/2n(n-1)}$. Next let $\beta\in F$ be such that $F=\IQ(\beta)$. Hence, $K=F(\alpha)=M(\beta)=\IQ(\alpha,\beta)$.
There are integers $0\leq a,b<d$ such that with $\gamma=a\alpha+b\beta$ we have $K=\IQ(\gamma)$ (see, e.g., \cite[Lemma 3.3]{art1}). 
As $p$ is unramified in $F$ and totally ramified in $M$ we conclude
that $[MF:F]=[M:\IQ]$, hence  $[K:F]=n$.
By standard properties of the height  we get
$H(\gamma)\leq 2d^2H(\alpha)H(\beta)$. Choosing the $\beta$ with minimal height we conclude
\begin{alignat}1\label{ineq:2}
\delta(K/\IQ)\leq 2^{1+\frac{1}{2n}}d^2\delta(F/\IQ)|\Delta_M|^{\frac{1}{2n(n-1)}}.
\end{alignat}
Finally, noticing that $[K:M]=m$ and thus $\Delta_M^m|\Delta_K$ the second claim follows from (\ref{ineq:2}).
\end{proof}
Theorem \ref{Thm:1} follows now easily from Proposition \ref{Prop:1}. 
Let $F=\IQ(2^{1/m})$ so that $[F:\IQ]=m$, $\delta(F/\IQ)\leq 2^{1/m}$, and $|\Delta_F|\leq (2m)^m$. 
Moreover, no prime $p>m$ ramifies in $F$.
The lower bound follows from part $(1)$. The upper bound follows from part $(2)$ upon noticing that $2^{1+\frac{1}{2n}}d^2\delta(F/\IQ)<(2d)^2$,
where for $m=1$ we have used that $\delta(F/\IQ)=1$.
This completes  the proof of Theorem \ref{Thm:1}.

\bibliographystyle{amsplain}
\bibliography{literature}

\end{document}